\documentclass[11pt]{amsart}
\usepackage[margin=1.1in]{geometry}
\usepackage{amsmath,amssymb,amsthm}
\usepackage[colorlinks=true,linkcolor=blue,citecolor=blue,urlcolor=blue]{hyperref}
\usepackage{booktabs}

\newtheorem{theorem}{Theorem}[section]
\newtheorem{lemma}[theorem]{Lemma}
\newtheorem{proposition}[theorem]{Proposition}
\newtheorem{corollary}[theorem]{Corollary}
\theoremstyle{definition}
\newtheorem{definition}[theorem]{Definition}
\newtheorem{example}[theorem]{Example}
\theoremstyle{remark}
\newtheorem{remark}[theorem]{Remark}

\newcommand{\F}{\mathbb{F}}
\newcommand{\Z}{\mathbb{Z}}
\newcommand{\Tr}{\mathrm{Tr}}
\newcommand{\Om}{\Omega}
\newcommand{\Omd}{\widehat{\Omega}}
\newcommand{\eq}{\mathbf{e}_q}
\newcommand{\er}{\mathbf{e}_r}
\newcommand{\ip}[2]{\langle #1,#2\rangle}
\newcommand{\KL}{\mathcal{K}}
\newcommand{\Kz}{K_0}
\newcommand{\Ko}{K_1}
\newcommand{\Kp}{K_0^{\perp}}

\title[Association schemes from vectorial gMM bent functions]{Association schemes from vectorial generalized Maiorana--McFarland bent functions with non-weakly regular components}

\author{Rumi Melih Pelen}
\address{Department of Mathematics and Statistics, University of South Florida, Tampa, FL, USA}
\email{rmpelen@usf.edu}

\date{September 26, 2026}

\begin{document}

\begin{abstract}
Let $p$ be an odd prime, $q=p^k$ and $r=p^s$ with $s\mid k$. We consider functions
$H\colon V_n\times\F_q\times\F_q\to\F_r$ of generalized Maiorana--McFarland type,
$H(x,y,z)=P^{(z)}(x)+\Tr^k_s(yz)$, whose ingredients $P^{(z)}\colon V_n\to\F_r$ are vectorial dual-bent functions with weakly regular components which are constant on the cosets of $\F_r^*$ in $\F_q^*$ and, for $z\ne0$, are vectorial dual-bent with respect to the inversion $t\mapsto t^{-1}$ of $\F_r^*$. For $s=k$ these are the $\F_q$-valued vectorial generalized Maiorana--McFarland functions of \c{C}e\c{s}melio\u{g}lu and Meidl, which can have all their components non-weakly regular. We show that the partition of $V_n\times\F_q\times\F_q$ into the $\F_r$-valued level sets of $H$, cut according to the $\F_r$-line of $z$ and with the subspace $\{0\}\times\F_q\times\{0\}$ cut according to the $\F_r$-line of $y$, is Fourier-reflexive (provided the ingredient $P^{(0)}$ at $z=0$ satisfies $P^{(0)}(0)=0$ and $W_{P^{(0)}_t}(0)=\pm p^{n/2}$ for all $t$), and induces a translation association scheme with $\frac{q-1}{r-1}(r+1)+r$ classes (one less in a degenerate case), which is symmetric when the ingredients are even. For $s=1$ this is (for $l=2$ and $n$ even) the recent construction of Anbar and Kalayc\i\ from $p$-ary generalized Maiorana--McFarland functions, the inversion condition being the $2$-form condition in that case; for $s=k$ one obtains $(2q+1)$-class schemes from vectorial bent functions with non-weakly regular components. Fusions along $\F_r$-subspaces $\Kz$ of $\F_q$, which need a compatibility condition on the Walsh signs, give $(2r+1)$- and $(3r+2)$-class schemes and, when all ingredients are of inversion type, connect the constructions at different levels; for $s=1$, $\Kz=\{0\}$ they are the $(2p+1)$-class schemes of Anbar, Kalayc\i, Meidl and \"Ozbudak. We give examples showing that the conclusion may fail when the inversion condition is omitted, when the ingredients vary within a coset of $\F_r^*$, when the cutting sets are not subspaces, or when the ingredients have non-weakly regular components; thus none of these hypotheses can simply be dropped. To our knowledge, these are the first association schemes obtained from vectorial bent functions with non-weakly regular components.
\end{abstract}

\maketitle

\section{Introduction}

Let $p$ be an odd prime and let $V_n$ be an $n$-dimensional vector space over $\F_p$ equipped with a nondegenerate symmetric bilinear form $\ip{\cdot}{\cdot}$. A function $f\colon V_n\to\F_p$ is \emph{bent} if its Walsh transform
\[
W_f(b)=\sum_{x\in V_n}\zeta_p^{\,f(x)-\ip{b}{x}},\qquad \zeta_p=e^{2\pi i/p},
\]
has absolute value $p^{n/2}$ for all $b\in V_n$. By \cite{KSW} one then has $W_f(b)=\varepsilon_f(b)\,\mu\, p^{n/2}\zeta_p^{f^*(b)}$ with $\varepsilon_f(b)\in\{\pm1\}$, $\mu\in\{1,i\}$ depending only on $p^n\bmod 4$, and a function $f^*\colon V_n\to\F_p$ called the \emph{dual} of $f$. The bent function $f$ is \emph{weakly regular} if $\varepsilon_f$ is constant, \emph{regular} if moreover $\varepsilon_f\mu=1$, and \emph{non-weakly regular} otherwise. The dual of a weakly regular bent function is bent; a non-weakly regular bent function may or may not be \emph{dual-bent} \cite{CMP13}.

A function $F\colon V_n\to V_m$ is \emph{vectorial bent} if all its component functions $F_c=\ip{c}{F}$, $c\ne0$, are bent. Vectorial bent functions with prescribed regularity properties of their components have been studied in \cite{CMP20,CM24} and the references therein. In \cite{CM24} the vectorial versions of the three known constructions of non-weakly regular bent functions are compared, and it is shown that the vectorial generalized Maiorana--McFarland construction yields vectorial bent functions all of whose components are non-weakly regular but dual-bent, as well as vectorial bent functions with both weakly regular and non-weakly regular components.

Bent functions and their level sets are a well-known source of partial difference sets, strongly regular graphs and association schemes. The classical results are for (vectorial) bent functions with weakly regular components \cite{PTFL,TPF,CMP18}; the preimage set partition of a weakly regular $p$-ary bent function which is an $\ell$-form induces a $p$-class association scheme \cite{PTFL,WHL}, and the preimage set partition of a \emph{vectorial dual-bent} function all of whose components are regular, or all weakly regular but not regular, induces a $p^m$-class association scheme \cite{AKMO3}, with extensions to general sign patterns in \cite{WFWY}; see also \cite{AKM,WF24}. For non-weakly regular bent functions the first construction of association schemes was given by \"Ozbudak and the author \cite{OP22} for ternary generalized Maiorana--McFarland functions; the basic idea is to intersect the level sets of the dual $f^*$ with the two sets $B_\pm(f)$ on which the Walsh transform of $f$ takes a fixed sign (equivalently, up to exchanging $f$ and $f^*$, the level sets of $f$ with $B_\pm(f^*)$). This was generalized by Wei, Wang and Fu \cite{WWF} to a large class of dual-bent non-weakly regular bent functions for arbitrary odd $p$, and by Anbar, Kalayc\i, Meidl and \"Ozbudak \cite{AKMO} to the whole class of $p$-ary generalized Maiorana--McFarland functions
\begin{equation}\label{eq:AKMO}
F(x,y,z)=f^{(z)}(x)+\Tr^k_1(yz^{l-1}),\qquad F\colon\F_{p^n}\times\F_{p^k}\times\F_{p^k}\to\F_p,
\end{equation}
which yield $(2p+1)$-class association schemes. In all these constructions the bent functions are $\F_p$-valued. Association schemes from vectorial bent functions with non-weakly regular components have, as far as we know, not been considered so far; whether they exist is asked explicitly at the end of \cite{AKMO3}: ``Can association schemes or similar objects also be obtained from non-ternary non-weakly regular bent functions, or from vectorial bent functions with mixed or non-weakly regular components?'' The present paper answers the second part of this question affirmatively.

In this paper we study functions
\begin{equation}\label{eq:H}
H\colon V_n\times\F_q\times\F_q\to\F_r,\qquad H(x,y,z)=P^{(z)}(x)+\Tr^k_s(yz),\qquad q=p^k,\ r=p^s,\ s\mid k,
\end{equation}
whose ingredients $P^{(z)}\colon V_n\to\F_r$ are vectorial bent functions with weakly regular components. For $s=k$, \eqref{eq:H} is the vectorial generalized Maiorana--McFarland function of \cite[Lemma~2]{CM24}, and for $s=1$ it is \eqref{eq:AKMO} with $l=2$. The $\F_r$-valued function $H$ has $r-1$ components $H_t=\Tr^s_1(tH)$, $t\in\F_r^*$; by \cite[Corollary~3]{CM24}, $H_t$ is non-weakly regular as soon as the components $\Tr^s_1(tP^{(z)})$ have different regularity for different $z$.

Our main result, Theorem~\ref{thm:main}, is the following. Suppose that the ingredients are constant on the cosets of $\F_r^*$ in $\F_q^*$, that they are vectorial dual-bent, and that those with $z\ne0$ are vectorial dual-bent with respect to the inversion of $\F_r^*$ (Definition~\ref{def:inv}). Then the partition of $V_n\times\F_q\times\F_q$ into the $\F_r$-valued level sets of $H$, cut according to the $\F_r$-line of $z$, and with the subspace $\{0\}\times\F_q\times\{0\}$ cut according to the $\F_r$-line of $y$, is Fourier-reflexive. It therefore induces a translation association scheme with $\frac{q-1}{r-1}(r+1)+r$ classes (one less in a degenerate case), symmetric when the ingredients are even, whatever the Walsh signs of the ingredients. Fusing along an $\F_r$-subspace $\Kz$ of $\F_q$ (Theorem~\ref{thm:K0}), which requires the ingredients to have the same sign pattern inside and outside $\Kz$, gives $(2r+1)$-class schemes for $\Kz\in\{\{0\},\F_q\}$ and $(3r+2)$-class schemes otherwise. Three instances deserve to be singled out.
\begin{itemize}
\item $s=1$: $H$ is the $p$-ary function \eqref{eq:AKMO} with $l=2$. Theorem~\ref{thm:main} is then, for $n$ even, the case $l=2$ of the recent theorem of Anbar and Kalayc\i\ \cite{AK}, which gives $\big(\frac{p^k-1}{p-1}(p+1)+p\big)$-class schemes, and Theorem~\ref{thm:K0} with $\Kz=\{0\}$ is the case $l=2$ of \cite[Theorem~1]{AKMO}; the inversion hypothesis is, for $s=1$, the $2$-form hypothesis of these papers (Proposition~\ref{prop:s1}).
\item $s=k$: $H$ is the $\F_q$-valued vectorial generalized Maiorana--McFarland function with two ingredients, all of whose components may be non-weakly regular, and Theorem~\ref{thm:main} gives $(2q+1)$-class schemes; when both ingredients are of inversion type, the schemes of \cite{AKMO} attached to the projections $\Tr^k_1(\gamma H)$, $\gamma\in\F_q^*$, are fusion schemes of these (Theorem~\ref{thm:fusion}).
\item $1<s<k$: the $\F_{p^s}$-valued analogues of the construction of \cite{AK}, which are neither fusions nor fissions of the schemes at other levels in general.
\end{itemize}
In all cases the class numbers stated above are one less in the degenerate case $n=2s$ with all components of the ingredient at $z=0$ weakly regular but not regular.
Along the way we show that, for two-ingredient functions, the inversion hypothesis on both ingredients is equivalent to $H$ being vectorial dual-bent (Proposition~\ref{prop:equiv}), so that in particular \emph{vectorial dual-bent generalized Maiorana--McFarland functions induce association schemes}, irrespective of the regularity of their components; but Theorem~\ref{thm:main} also applies to functions $H$ which are not vectorial dual-bent, since the ingredient on $z=0$ only needs to be vectorial dual-bent with an arbitrary permutation. For the change of level in Theorem~\ref{thm:fusion} the inversion hypothesis on the ingredient at $z=0$ is needed again. Explicit infinite families, built from quadratic and Maiorana--McFarland ingredients, are given in Corollary~\ref{cor:family}.

Compared with the scalar constructions of \cite{AKMO,AK}, the refinement to $\F_r$-valued level sets costs flexibility, and we give examples (Section~\ref{sec:examples}) showing that the conclusion may fail when any of the hypotheses is omitted: with vectorial partial spread ingredients (which are not of inversion type), with ingredients constant only on the smaller cosets of $\F_{p}^*$ (which give schemes at level $1$ but not at level $k$), with cutting sets which are not subspaces, and with ingredients having non-weakly regular components. Thus none of the hypotheses can simply be dropped in general. On the other hand no assumption is needed on the Walsh signs themselves, so $H$ may have only non-weakly regular components, only weakly regular components, or both.

The paper is organized as follows. Section~\ref{sec:prelim} collects the background. Section~\ref{sec:H} determines the Walsh transform and the vectorial dual of $H$. Section~\ref{sec:main} contains the main theorem, Section~\ref{sec:fusion} the fusion schemes obtained by changing the level $s$, and Section~\ref{sec:examples} examples, computational verifications and the counterexamples.

\section{Preliminaries}\label{sec:prelim}

Throughout, $p$ is an odd prime, $q=p^k$, $s\mid k$, $r=p^s$, so $\F_r\subseteq\F_q$, and $\zeta_p=e^{2\pi i/p}$. We write $\eq(u)=\zeta_p^{\Tr^k_1(u)}$ for $u\in\F_q$ and $\er(u)=\zeta_p^{\Tr^s_1(u)}$ for $u\in\F_r$; note that $\eq(tu)=\er(t\Tr^k_s(u))$ for $t\in\F_r$, $u\in\F_q$. For a vector space $V_n$ over $\F_p$ with nondegenerate symmetric bilinear form $\ip{\cdot}{\cdot}$ and $f\colon V_n\to\F_p$ we use $W_f(b)=\sum_{x\in V_n}\zeta_p^{f(x)-\ip{b}{x}}$. We assume throughout that $n$ is even, so that $p^n\equiv1\bmod4$ and
\begin{equation}\label{eq:walsh}
W_f(b)=\varepsilon_f(b)\,p^{n/2}\,\zeta_p^{f^*(b)},\qquad \varepsilon_f(b)\in\{\pm1\},
\end{equation}
for every bent $f$. For $n$ even, the dual of a weakly regular bent function $f$ with sign $\varepsilon$ is weakly regular bent with the same sign: $W_{f^*}(b)=\varepsilon p^{n/2}\zeta_p^{f(-b)}$.

\subsection{Vectorial dual-bent functions}
A function $P\colon V_n\to\F_r$ is vectorial bent if $P_t=\Tr^s_1(tP)$ is bent for every $t\in\F_r^*$; it has \emph{weakly regular components} if every $P_t$ is weakly regular, with sign $\varepsilon_{P,t}$. The vector $(\varepsilon_{P,t})_{t\in\F_r^*}$ is the \emph{sign pattern} of $P$; the components are \emph{of the same type} if the sign pattern is constant. Following \cite{CMP20,WF24}, a vectorial bent $P\colon V_n\to\F_r$ is \emph{vectorial dual-bent} if there are a vectorial bent $P^*\colon V_n\to\F_r$ and a permutation $\sigma$ of $\F_r^*$ with $(P_t)^*=(P^*)_{\sigma(t)}$ for all $t$. The pair $(P^*,\sigma)$ is not unique: for an $\F_p$-linear permutation $A$ of $\F_r$ with adjoint $A^\top$ with respect to the trace form, $\Tr^s_1(A(u)v)=\Tr^s_1(uA^\top(v))$, the pair $(A^\top\circ P^*,A^{-1}\circ\sigma)$ serves as well.

\begin{definition}\label{def:inv}
A vectorial bent function $P\colon V_n\to\F_r$ is \emph{of inversion type} if there is a function $G\colon V_n\to\F_r$ with
\[
(P_t)^*=G_{t^{-1}}=\Tr^s_1(t^{-1}G)\qquad\text{for all } t\in\F_r^*.
\]
If $P$ has weakly regular components, then $G$ is vectorial bent and $P$ is vectorial dual-bent with $P^*=G$, $\sigma(t)=t^{-1}$; by the preceding remark $P$ is of inversion type whenever it is vectorial dual-bent with $\sigma=A\circ\iota$, $\iota(t)=t^{-1}$, for some $\F_p$-linear permutation $A$.
\end{definition}

\begin{example}\label{ex:ingredients}
\begin{enumerate}
\item[(a)] Let $s\mid n$, $V_n=\F_{p^n}$ with $\ip{b}{x}=\Tr^n_1(bx)$, and $P(x)=\Tr^n_s(\lambda x^2)$ with $\lambda\in\F_r^*$. For $t\in\F_r^*$, $P_t(x)=\Tr^n_1(t\lambda x^2)$ is quadratic bent with $W_{P_t}(b)=\eta(t\lambda)(-1)^{n-1}p^{n/2}\zeta_p^{-\Tr^n_1(b^2/(4t\lambda))}$ if $p\equiv1\bmod4$ and $W_{P_t}(b)=\eta(t\lambda)(-1)^{n-1}i^n p^{n/2}\zeta_p^{-\Tr^n_1(b^2/(4t\lambda))}$ if $p\equiv3\bmod4$, $\eta$ the quadratic character of $\F_{p^n}$ \cite[Corollary~3]{HK}. Hence $(P_t)^*=\Tr^s_1(t^{-1}G)$ with $G(b)=-\Tr^n_s(b^2/(4\lambda))$: $P$ is of inversion type with weakly regular components and sign pattern $\varepsilon_{P,t}=\eta(t\lambda)(-1)^{n-1}$, resp.\ $\eta(t\lambda)(-1)^{n-1}i^n$. If $n/s$ is even, then $\eta(t\lambda)=1$ for all $t\in\F_r^*$ and the components are of the same type; if $n/s$ is odd, then $\eta(t\lambda)=\eta_r(t\lambda)$ and both types occur.
\item[(b)] Let $n=2m$, $V_n=\F_{p^m}\times\F_{p^m}$ with $\ip{(a,b)}{(x,y)}=\Tr^m_1(ax+by)$, $s\mid m$, and $L$ an $\F_r$-linear permutation of $\F_{p^m}$. The vectorial Maiorana--McFarland function $P(x,y)=\Tr^m_s(L(x)y)$ has regular components $P_t(x,y)=\Tr^m_1(tL(x)y)$ with dual $(P_t)^*(a,b)=-\Tr^m_1(aL^{-1}(t^{-1}b))=\Tr^s_1(t^{-1}G(a,b))$, $G(a,b)=-\Tr^m_s(aL^{-1}(b))$. Hence $P$ is of inversion type.
\item[(c)] For $s=1$ and $p=3$ every even bent $f\colon V_n\to\F_3$ is of inversion type: $(2f)^*=(-f)^*=-f^*$ because $f^*$ is even. For $p>3$ this is no longer automatic.
\end{enumerate}
\end{example}

In contrast, the vectorial partial spread function $P(x,y)=\Tr^m_s(xy^{-1})$, $P(x,0)=0$, all of whose components are regular, is vectorial dual-bent with $\sigma(t)=t$ \cite[Theorem~2]{AKMO2}, and for $r>3$ it is not of inversion type (Example~\ref{ex:PS}).

For $s=1$ the notions above reduce to the familiar $\ell$-forms. Recall that $f\colon V_n\to\F_p$ is an \emph{$\ell$-form}, $2\le\ell\le p-1$ with $\gcd(\ell-1,p-1)=1$, if $f(cx)=c^{\ell}f(x)$ for all $c\in\F_p^*$; if $f$ is weakly regular bent and an $\ell$-form, then $f^*$ is an $\tilde\ell$-form with $(\ell-1)(\tilde\ell-1)\equiv1\bmod(p-1)$ \cite{AKMO}.

\begin{proposition}\label{prop:s1}
Let $n$ be even and $f\colon V_n\to\F_p$ weakly regular bent. Then $(tf)^*(b)=t\,f^*(t^{-1}b)$ for all $t\in\F_p^*$. Consequently $f$ is vectorial dual-bent (as an $\F_p$-valued function) if and only if $f^*$, equivalently $f$, is an $\ell$-form for some $\ell$; and $f$ is of inversion type if and only if $f$ is a $2$-form.
\end{proposition}
\begin{proof}
Let $\sigma_t$ be the automorphism of $\mathbb Q(\zeta_p)$ with $\sigma_t(\zeta_p)=\zeta_p^t$. Then $W_{tf}(b)=\sum_x\zeta_p^{t(f(x)-\ip{t^{-1}b}{x})}=\sigma_t\big(W_f(t^{-1}b)\big)=\varepsilon p^{n/2}\zeta_p^{tf^*(t^{-1}b)}$, as $\varepsilon p^{n/2}\in\mathbb Q$. If $f^*$ is an $\tilde\ell$-form then $(tf)^*=t^{1-\tilde\ell}f^*$, so $f$ is vectorial dual-bent with $\sigma(t)=t^{1-\tilde\ell}$; conversely $(tf)^*=\sigma(t)f^*$ gives $f^*(t^{-1}b)=t^{-1}\sigma(t)f^*(b)$, so $\chi(c):=c\,\sigma(c^{-1})$ satisfies $f^*(cb)=\chi(c)f^*(b)$ for all $b$; evaluating at some $b$ with $f^*(b)\ne0$ shows that $\chi$ is a character of $\F_p^*$, $\chi(c)=c^{\tilde\ell}$ for some $\tilde\ell$, so $f^*$ is an $\tilde\ell$-form, and $\sigma(t)=t^{1-\tilde\ell}$ being a permutation forces $\gcd(\tilde\ell-1,p-1)=1$. Since $f(x)=f^{**}(-x)$, $f$ is then an $\ell$-form with $(\ell-1)(\tilde\ell-1)\equiv1\bmod(p-1)$. The last statement is the case $\sigma=\iota$, i.e.\ $\tilde\ell=2$, which is equivalent to $\ell=2$.
\end{proof}

The first equivalence is \cite[Subsection~2.1]{AKMO}, see also \cite[Theorem~5]{WF24} and \cite[Remark~6]{AKMO3}. For $s>1$ the Galois argument is not available, and the inversion condition is a genuinely vectorial one; cf.\ the remark in \cite[Section~4]{AKMO3} that for $m>1$ vectorial dual-bentness is not equivalent to being a vectorial $\ell$-form.

\subsection{Association schemes from partitions}
Let $V$ be a finite dimensional $\F_p$-space with a nondegenerate symmetric bilinear form, and for $\xi\in V$, $S\subseteq V$ let $\chi_\xi(S)=\sum_{s\in S}\zeta_p^{-\ip{\xi}{s}}$. (The convention with $+\ip{\xi}{s}$ differs by the bijection $\xi\mapsto-\xi$, which changes neither the dual partition up to relabelling nor Fourier-reflexivity.) For a partition $\Om=\{U_0=\{0\},U_1,\dots,U_d\}$ of $V$, the \emph{dual partition} $\Omd$ is the partition of $V$ into the level sets of $\xi\mapsto(\chi_\xi(U_0),\dots,\chi_\xi(U_d))$, and $\Om$ is \emph{Fourier-reflexive} if $|\Omd|=|\Om|$ \cite{GL,BCN}; see \cite[Lemma~1]{AKMO}.

\begin{lemma}\label{lem:reflexive}
If $\Om=\{\{0\},U_1,\dots,U_d\}$ is Fourier-reflexive, then $R_0=\{(x,x)\}$, $R_i=\{(x,y): y-x\in U_i\}$ form a $d$-class (translation) association scheme on $V$, which is symmetric if $U_i=-U_i$ for all $i$.
\end{lemma}

\begin{lemma}\label{lem:injective}
For $Y,i\in\F_r$, $\varepsilon\colon\F_r^*\to\{\pm1\}$ and a permutation $\pi$ of $\F_r^*$ let $\KL^{\pi}_\varepsilon(Y,i)=\sum_{t\in\F_r^*}\varepsilon(t)\,\er(\pi(t)Y-ti)$, and write $\KL_\varepsilon=\KL^{\iota}_\varepsilon$ for $\pi=\iota\colon t\mapsto t^{-1}$. If $\KL^\pi_\varepsilon(Y,i)=\KL^\pi_\varepsilon(Y',i)$ for all $i\in\F_r$, then $Y=Y'$.
\end{lemma}
\begin{proof}
$i\mapsto\KL^\pi_\varepsilon(Y,i)$ is the Fourier transform on $\F_r$ of $\varphi_Y(t)=\varepsilon(t)\er(\pi(t)Y)$ ($t\ne0$), $\varphi_Y(0)=0$. Equal transforms give $\varphi_Y=\varphi_{Y'}$, i.e.\ $\er(\pi(t)(Y-Y'))=1$ for all $t\in\F_r^*$; as $\pi(t)$ runs through $\F_r^*$, $Y=Y'$.
\end{proof}

\section{The function $H$ and its vectorial dual}\label{sec:H}

Fix an $\F_r$-subspace $\Kz$ of $\F_q$, put $\Ko=\F_q\setminus\Kz$ and
\[
\Kp=\{y\in\F_q:\Tr^k_1(zy)=0\text{ for all }z\in\Kz\},
\]
an $\F_r$-subspace with $|\Kz||\Kp|=q$ and $(\Kp)^\perp=\Kz$. Note that $\Kz$ and $\Ko$ are stable under multiplication by $\F_r^*$. Let $(P^{(z)})_{z\in\F_q}$ be a family of functions $V_n\to\F_r$ satisfying
\begin{enumerate}
\item[(F1)] $P^{(cz)}=P^{(z)}$ for all $z\in\F_q$ and $c\in\F_r^*$;
\item[(F2)] every $P^{(z)}$ is vectorial bent with weakly regular components, with sign pattern $\varepsilon_{z,t}$; for $z\ne0$, $P^{(z)}$ is of inversion type, $(P^{(z)}_t)^*=\Tr^s_1(t^{-1}G^{(z)})$; and $P^{(0)}$ is vectorial dual-bent, $(P^{(0)}_t)^*=\Tr^s_1(\sigma_0(t)G^{(0)})$ for some permutation $\sigma_0$ of $\F_r^*$ (in Lemma~\ref{lem:walsh} and Proposition~\ref{prop:equiv} we additionally assume $\sigma_0=\iota$ where stated);
\item[(F3)] (with respect to $\Kz$) $\varepsilon_{z,t}=\varepsilon_{1,t}$ for all $z\in\Ko$, and $\varepsilon_{z,t}=\varepsilon_{2,t}$ for all $z\in\Kz\setminus\{0\}$, where $\varepsilon_{1,\cdot},\varepsilon_{2,\cdot}$ are fixed sign patterns (the second condition is void if $\Kz=\{0\}$, the first if $\Kz=\F_q$); this condition is used only in Theorem~\ref{thm:K0};
\item[(F4)] $P^{(0)}(0)=0$ and $G^{(0)}(0)=0$, i.e.\ $W_{P^{(0)}_t}(0)=\varepsilon_{0,t}p^{n/2}$ for all $t$.
\end{enumerate}
By (F1), also $G^{(cz)}=G^{(z)}$ and $\varepsilon_{cz,t}=\varepsilon_{z,t}$ for $c\in\F_r^*$. The asymmetry between $z=0$ and $z\ne0$ in (F2) is genuine: the inversion is forced by the term $\Tr^k_s(yz)$, which is absent on $z=0$; see Lemma~\ref{lem:walsh}, Theorem~\ref{thm:main} and Example~\ref{ex:PS}. Define $H\colon V=V_n\times\F_q\times\F_q\to\F_r$ by \eqref{eq:H}, equip $V$ with $\ip{(u,v,w)}{(x,y,z)}=\ip{u}{x}+\Tr^k_1(vy)+\Tr^k_1(wz)$, and put $Q=p^n$. The simplest case, to which the reader may restrict on first reading, is that of two or three ingredients: $P^{(z)}=P^{(0)}$ for $z=0$, $P^{(z)}=P^{(2)}$ for $z\in\Kz\setminus\{0\}$, $P^{(z)}=P^{(1)}$ for $z\in\Ko$; then (F1) and (F3) are automatic.

\begin{lemma}\label{lem:walsh}
Suppose that also $P^{(0)}$ is of inversion type ($\sigma_0=\iota$). For $t\in\F_r^*$ and $(u,v,w)\in V$,
\[
W_{H_t}(u,v,w)=q\,\eq(-t^{-1}vw)\,W_{P^{(v)}_t}(u)=\varepsilon_{v,t}\,(Qq^2)^{1/2}\,\er\!\big(t^{-1}G^{(v)}(u)-t^{-1}\Tr^k_s(vw)\big).
\]
In particular $H$ is vectorial bent, the Walsh transform of $H_t$ has sign $\varepsilon_{v,t}$ at $(u,v,w)$, which depends on $v$ only through the coset $v\F_r^*$ (so, under (F3), it is $\varepsilon_{0,t}$ on $\{v=0\}$, $\varepsilon_{2,t}$ on $\{v\in\Kz\setminus\{0\}\}$ and $\varepsilon_{1,t}$ on $\{v\in\Ko\}$), and
\begin{equation}\label{eq:dualH}
(H_t)^*=\Tr^s_1(t^{-1}K),\qquad K(u,v,w)=G^{(v)}(u)-\Tr^k_s(vw).
\end{equation}
Thus $H$ is vectorial dual-bent of inversion type with vectorial dual $K$, which is again of the form \eqref{eq:H} with ingredients $G^{(z)}$ (and $-\Tr^k_s(yz)$ in place of $\Tr^k_s(yz)$).
\end{lemma}
\begin{proof}
Summing over $y$ first, $\sum_y\eq((tz-v)y)=q$ if $z=t^{-1}v$ and $0$ otherwise, so
\[
W_{H_t}(u,v,w)=q\sum_x\zeta_p^{P^{(t^{-1}v)}_t(x)-\ip{u}{x}}\eq(-t^{-1}vw)=q\,\eq(-t^{-1}vw)W_{P^{(v)}_t}(u),
\]
using $P^{(t^{-1}v)}=P^{(v)}$ by (F1). The rest follows from \eqref{eq:walsh}, (F2) and $\eq(t^{-1}vw)=\er(t^{-1}\Tr^k_s(vw))$.
\end{proof}

\begin{corollary}\label{cor:nwr}
The component $H_t$ is weakly regular if $\varepsilon_{z,t}$ is the same for all $z\in\F_q$, and non-weakly regular otherwise. For every $t$, the Walsh transform of $H_t$ has constant sign on each of the blocks $V_n\times\{0\}\times\F_q$ and $V_n\times(L\setminus\{0\})\times\F_q$, $L$ an $\F_r$-line of $\F_q$; the block structure does not depend on $t$, while the signs attached to the blocks may.
\end{corollary}

If $P^{(0)}$ is not of inversion type, Lemma~\ref{lem:walsh} remains valid on $\{v\ne0\}$, while on $\{v=0\}$ one has $W_{H_t}(u,0,w)=q\,W_{P^{(0)}_t}(u)=\varepsilon_{0,t}(Qq^2)^{1/2}\er(\sigma_0(t)G^{(0)}(u))$; in particular $H$ is still vectorial bent with the same sign behaviour, but in the two- or three-ingredient case $H$ is then not vectorial dual-bent, by the following proposition. (Recall that the pair $(P^{*},\sigma)$ is not unique; ``not of inversion type'' means that no choice of the pair has $\sigma=\iota$.) For two or three ingredients the inversion hypothesis is not only sufficient but necessary for $H$ to be vectorial dual-bent.

\begin{proposition}\label{prop:equiv}
Suppose $P^{(z)}$ is constant on each of $\{0\}$, $\Kz\setminus\{0\}$, $\Ko$, with values $P^{(0)},P^{(2)},P^{(1)}$, all vectorial bent with weakly regular components. Then $H$ is vectorial dual-bent if and only if those of $P^{(0)},P^{(2)},P^{(1)}$ which occur (i.e.\ whose region $\{0\}$, $\Kz\setminus\{0\}$, $\Ko$ is nonempty) are of inversion type. In that case every pair $(H^*,\sigma)$ as in the definition of vectorial dual-bentness is $\big((A^\top)^{-1}\circ K,\ A\circ\iota\big)$, equivalently $(A^\top\circ K,\ A^{-1}\circ\iota)$, with $K$ as in \eqref{eq:dualH} and $A$ an $\F_p$-linear permutation of $\F_r$.
\end{proposition}
\begin{proof}
``If'' is Lemma~\ref{lem:walsh}. Conversely let $(H_t)^*=\Tr^s_1(\sigma(t)K')$. By the proof of Lemma~\ref{lem:walsh} (which does not use (F2)), $(H_t)^*(u,v,w)=(P^{(v)}_t)^*(u)-\Tr^s_1(t^{-1}\Tr^k_s(vw))$. Comparing at $(u,v,w)$ and $(u,v,0)$ gives, for $v\ne0$,
\[
\Tr^s_1\big(\sigma(t)D\big)=\Tr^s_1(t^{-1}c),\qquad D=K'(u,v,0)-K'(u,v,w),\ c=\Tr^k_s(vw),
\]
for all $t\in\F_r^*$. As $\sigma(t)$ runs through $\F_r^*$ and the trace form is nondegenerate, $D$ is determined by $c$; write $D=D_c$. Since $c\mapsto\Tr^s_1(t^{-1}c)$ is additive, uniqueness gives $D_{c+c'}=D_c+D_{c'}$, so $D$ is an $\F_p$-linear map $\F_r\to\F_r$, injective because $D_c=0$ forces $c=0$. With $B$ the adjoint of $D$ we get $\Tr^s_1(B(\sigma(t))c)=\Tr^s_1(t^{-1}c)$ for all $c$, i.e.\ $\sigma(t)=A(t^{-1})$ with $A=B^{-1}$. Hence $(H_t)^*=\Tr^s_1(t^{-1}A^\top K')$, and restricting to $w=0$ and to $v$ in $\{0\}$, $\Kz\setminus\{0\}$, $\Ko$ shows that those of $P^{(0)},P^{(2)},P^{(1)}$ which occur are of inversion type (the functions $A^\top K'(u,0,w)$, resp.\ $A^\top K'(u,v,0)$, do not depend on $w$, resp.\ on $v$ within a region, by nondegeneracy of the trace form). Finally $A^\top K'$ and $K$ have the same components, hence $K'=(A^\top)^{-1}K$.
\end{proof}

\section{The main theorem}\label{sec:main}

We keep the setting of Section~\ref{sec:H}, but we no longer fix $\Kz$; in this section (F3) is not needed. Let $\mathcal L$ be the set of one-dimensional $\F_r$-subspaces (\emph{$\F_r$-lines}) of $\F_q$, so $|\mathcal L|=N:=(q-1)/(r-1)$, the sets $L\setminus\{0\}$, $L\in\mathcal L$, are the cosets of $\F_r^*$ in $\F_q^*$, and $L^\perp=\{y\in\F_q:\Tr^k_1(zy)=0\ \forall z\in L\}$ is a hyperplane with $|L^\perp|=q/r$. By (F1), $P^{(z)}$, $G^{(z)}$ and $\varepsilon_{z,\cdot}$ are constant on $L\setminus\{0\}$; we write $P^{(L)}$, $G^{(L)}$, $\varepsilon_{L,\cdot}$ and $\KL^{(L)}=\KL_{\varepsilon_{L,\cdot}}$ (with $\pi=\iota$), and $\KL^{(0)}=\KL^{\sigma_0}_{\varepsilon_{0,\cdot}}$. (The superscript distinguishes these sums from the subspace $\Kz$ of Section~\ref{sec:H}.) For $L\in\mathcal L$ and $i\in\F_r$ define the following subsets of $V$:
\begin{align*}
c_{L,i}&=\{(x,y,z): z\in L\setminus\{0\},\ H(x,y,z)=i\},\qquad B_L=\{(0,y,0): y\in L\setminus\{0\}\},\\
e_i&=\{(x,y,0): x\ne0,\ P^{(0)}(x)=i\},
\end{align*}
and let
\begin{equation}\label{eq:Omegafull}
\Om_s(H)=\{\{0\},\ e_i\ (i\in\F_r),\ c_{L,i}\ (L\in\mathcal L,\ i\in\F_r),\ B_L\ (L\in\mathcal L)\},
\end{equation}
with $e_0$ omitted if it is empty. Thus $\Om_s(H)$ is the partition of $V$ into the $\F_r$-valued level sets of $H$, cut according to the $\F_r$-line of $z$, with $\{0\}\times\F_q\times\{0\}$ cut according to the $\F_r$-line of $y$. For $s=k$ there is only one line, and $\Om_k(H)=\{\{0\},e_i,c_i,B\setminus\{0\}\}$ with $c_i=\{z\ne0,H=i\}$, $B=\{0\}\times\F_q\times\{0\}$ (this is the partition $\Om(H,\{0\})$ of \eqref{eq:Omega} below); by Corollary~\ref{cor:nwr} applied to $K$, the sets $c_i$ and $e_i\cup B$ are then the level sets of $H$ intersected with the two sign sets $B_\pm(K)$ of the vectorial dual $K=H^*$; they are the vectorial analogues of the sets $C_i(F)=B_+(F)\cap D_{F^*,i}$, $D_i(F)=B_-(F)\cap D_{F^*,i}$ of \cite{OP22} (with $F=K$), respectively of the sets $c_i(f)=\{x\in B_+(f^*):f(x)=i\}$, $d_i(f)=\{x\in B_-(f^*):f(x)=i\}$ of \cite{WWF}. For $s=1$, $\Om_1(H)$ is the partition of \cite{AK} (see Section~\ref{sec:fusion}).

For $b\in\F_q^*$ and $\xi=(a,b,c)\in V$ put $Y_b(\xi)=G^{(b)}(a)-\Tr^k_s(bc)\in\F_r$, and write $[\cdot]$ for the indicator of a condition.

\begin{lemma}\label{lem:charsums}
Let $\xi=(a,b,c)\in V$, $L\in\mathcal L$ and $i\in\F_r$. Then
\begin{align}
\chi_\xi(c_{L,i})&=[b\in L\setminus\{0\}]\,\tfrac qr\,p^{n/2}\KL^{(L)}\big(Y_b(\xi),i\big)+\delta_{a,0}\delta_{b,0}\,\tfrac qr\,Q\big(r[c\in L^\perp]-1\big),\label{eq:chic}\\
\chi_\xi(e_i)&=\delta_{b,0}\,\tfrac qr\Big(Q\delta_{a,0}+p^{n/2}\KL^{(0)}\big(G^{(0)}(a),i\big)\Big)-q\,\delta_{b,0}\delta_{i,0},\label{eq:chie}\\
\chi_\xi(B_L)&=r[b\in L^\perp]-1.\label{eq:chiB}
\end{align}
\end{lemma}
\begin{proof}
Writing the indicator of $\{H=i\}$ as $r^{-1}\sum_{t\in\F_r}\er(t(H-i))$ and using $\er(t\Tr^k_s(yz))=\eq(tyz)$,
\[
\chi_\xi(c_{L,i})=\frac1r\sum_{t\in\F_r}\er(-ti)\sum_{z\in L\setminus\{0\}}\eq(-cz)\sum_{x\in V_n}\zeta_p^{P^{(z)}_t(x)-\ip{a}{x}}\sum_{y\in\F_q}\eq\big((tz-b)y\big).
\]
The inner sum is $q$ if $tz=b$ and $0$ otherwise. For $t=0$ this forces $b=0$, and the contribution is $\frac qr\delta_{b,0}\,Q\delta_{a,0}\sum_{z\in L\setminus\{0\}}\eq(-cz)=\frac qr\delta_{a,0}\delta_{b,0}Q\big(r[c\in L^\perp]-1\big)$. For $t\ne0$ we need $z=b/t\in L\setminus\{0\}$, i.e.\ $b\in L\setminus\{0\}$ (as $L$ is $\F_r^*$-stable), and then $P^{(b/t)}=P^{(L)}$ by (F1), so the contribution is
\[
\frac qr\sum_{t\ne0}\er(-ti)\,\eq(-cbt^{-1})\,W_{P^{(L)}_t}(a)=\frac qr\,p^{n/2}\sum_{t\ne0}\varepsilon_{L,t}\,\er\big(t^{-1}(G^{(L)}(a)-\Tr^k_s(bc))-ti\big),
\]
by \eqref{eq:walsh} and (F2). This is \eqref{eq:chic}. For $e_i$, the sum over $y$ gives $q\delta_{b,0}$, and
\[
\sum_{x\ne0,\,P^{(0)}(x)=i}\zeta_p^{-\ip{a}{x}}=\frac1r\sum_{t\in\F_r}\er(-ti)\Big(\sum_{x}\zeta_p^{P^{(0)}_t(x)-\ip{a}{x}}-1\Big)
=\frac1r\Big(Q\delta_{a,0}+p^{n/2}\KL^{(0)}(G^{(0)}(a),i)-r\delta_{i,0}\Big),
\]
where $W_{P^{(0)}_t}(a)=\varepsilon_{0,t}p^{n/2}\er(\sigma_0(t)G^{(0)}(a))$ by (F2); this gives \eqref{eq:chie}, and \eqref{eq:chiB} is immediate.
\end{proof}

\begin{theorem}\label{thm:main}
Let $n$ be even with $s\le n/2$, and let $(P^{(z)})_{z\in\F_q}$ satisfy (F1), (F2), (F4). Let $H$ be given by \eqref{eq:H} and $\Om_s(H)$ by \eqref{eq:Omegafull}. Then $\Om_s(H)$ is Fourier-reflexive, and its dual partition is $\Om_s(K)$ for the function $K(u,v,w)=G^{(v)}(u)-\Tr^k_s(vw)$ of \eqref{eq:dualH} (which is the vectorial dual of $H$ when $\sigma_0=\iota$), with the roles of the second and third coordinates interchanged. Consequently, if all $P^{(z)}$ are even, $\Om_s(H)$ induces a symmetric association scheme on $V$ with
\[
|\Om_s(H)|-1=\frac{q-1}{r-1}\,(r+1)+r-1+\delta\quad\text{classes},\qquad
\delta=\begin{cases}0&\text{if } n=2s\text{ and }\varepsilon_{0,t}=-1\text{ for all }t\in\F_r^*,\\ 1&\text{otherwise.}\end{cases}
\]
In particular, for $s=k$ one obtains $2q+\delta$ classes.
\end{theorem}
\begin{proof}
By Lemma~\ref{lem:charsums} the vector $(\chi_\xi(U))_{U\in\Om_s(H)}$ depends on $\xi=(a,b,c)$ as follows.
\begin{enumerate}
\item[(i)] $b\ne0$, say $b\in L_0\setminus\{0\}$: $\chi_\xi(e_i)=0$, $\chi_\xi(c_{L,i})=0$ for $L\ne L_0$, $\chi_\xi(c_{L_0,i})=\frac qr p^{n/2}\KL^{(L_0)}(Y_b(\xi),i)$, and $\chi_\xi(B_L)=r[b\in L^\perp]-1$. Since $b\in L^\perp$ if and only if $L\subseteq\{y\in\F_q:\Tr^k_s(by)=0\}$, the vector $([b\in L^\perp])_{L\in\mathcal L}$ determines this $\F_r$-hyperplane, whose $\F_r$-orthogonal complement is $L_0$; so it determines $L_0$. For fixed $L_0$ the vector depends on $\xi$ only through $Y_b(\xi)$, injectively by Lemma~\ref{lem:injective}. Every $Y\in\F_r$ occurs, since $\Tr^k_s(bc)$ runs through $\F_r$ as $c$ does. So this region consists of exactly $Nr$ classes, indexed by $(L_0,Y)$.
\item[(ii)] $b=0$, $a\ne0$: $\chi_\xi(c_{L,i})=0$, $\chi_\xi(B_L)=r-1$, $\chi_\xi(e_i)=\frac qr p^{n/2}\KL^{(0)}(G^{(0)}(a),i)-q\delta_{i,0}$. The vector depends injectively on $G^{(0)}(a)$; the classes are $\{(a,0,c):a\ne0,\ G^{(0)}(a)=Y\}$, nonempty for $Y\ne0$ (by the count \eqref{eq:count} below applied to $G^{(0)}$, whose components are weakly regular bent with $W_{G^{(0)}_t}(0)=\pm p^{n/2}$, see the end of the proof), and for $Y=0$ nonempty iff $G^{(0)}$ has a nonzero zero; write $\delta'\in\{0,1\}$ for this.
\item[(iii)] $a=b=0$, $c\ne0$: $\chi_\xi(c_{L,i})=\frac qr Q(r[c\in L^\perp]-1)$, $\chi_\xi(e_i)=\frac qr(Q+p^{n/2}\KL^{(0)}(0,i))-q\delta_{i,0}$, $\chi_\xi(B_L)=r-1$. As in (i), the vector $([c\in L^\perp])_L$ determines the $\F_r$-line of $c$, so this region splits into the $N$ classes $\{(0,0,c): c\in L\setminus\{0\}\}$.
\item[(iv)] $\xi=0$.
\end{enumerate}
These classes are distinct: (i) is separated from (ii)--(iv) by the vector $(\chi_\xi(B_L))_L$, which is constant $r-1$ only for $b=0$; (ii) from (iii) and (iv) by $\chi_\xi(c_{L,i})$, which vanishes for all $L$ in (ii), equals $-\frac qrQ\ne0$ for some $L$ in (iii) (a line not contained in the $\F_r$-hyperplane of $c$), and equals $|c_{L,i}|>0$ in (iv); and (iii) from (iv) by the same values. Hence $|\widehat{\Om_s(H)}|=1+Nr+(r-1+\delta')+N$. On the other hand $c_{L,i}\ne\emptyset$ (choose $y$), $B_L\ne\emptyset$, $|e_i|=q\,(|\{x:P^{(0)}(x)=i\}|-\delta_{i,0})$, and by \eqref{eq:walsh} and (F4)
\begin{equation}\label{eq:count}
|\{x:P^{(0)}(x)=i\}|=\frac1r\sum_{t\in\F_r}W_{P^{(0)}_t}(0)\er(-ti)=p^{n-s}+p^{n/2-s}\KL^{(0)}(0,i),
\end{equation}
which is positive for $i\ne0$ since $|\KL^{(0)}(0,i)|\le r-1<p^{n/2}$ by $s\le n/2$. (If the components of $P^{(0)}$ are of the same type, $s\le n/2$ is automatic by Nyberg's bound, see \cite[Section~1]{CM24}; for mixed sign patterns it is a genuine hypothesis, as the planar function $x^2$ on $\F_{p^n}$, with $s=n$, shows.) So $e_i\ne\emptyset$ for $i\ne0$, $e_0\ne\emptyset$ iff $\delta:=[P^{(0)}\text{ has a nonzero zero}]=1$, and $|\Om_s(H)|=1+Nr+(r-1+\delta)+N$. It remains to show $\delta=\delta'$ and to evaluate $\delta$. By \eqref{eq:count}, $|\{x:P^{(0)}(x)=0\}|=p^{n-s}+p^{n/2-s}\sum_{t\ne0}\varepsilon_{0,t}$. The components of $G^{(0)}$ are $G^{(0)}_t=(P^{(0)}_{\sigma_0^{-1}(t)})^*$, weakly regular bent with sign $\varepsilon_{0,\sigma_0^{-1}(t)}$ and $W_{G^{(0)}_t}(0)=\varepsilon_{0,\sigma_0^{-1}(t)}p^{n/2}\zeta_p^{P^{(0)}_{\sigma_0^{-1}(t)}(0)}=\varepsilon_{0,\sigma_0^{-1}(t)}p^{n/2}$ by (F4); hence $|\{a:G^{(0)}(a)=0\}|$ is the same number. This number is at least $p^{n/2-s}(p^{n/2}-r+1)>1$ unless $n=2s$, and for $n=2s$ it equals $r+\sum_{t\ne0}\varepsilon_{0,t}$, which is $1$ iff all $\varepsilon_{0,t}=-1$. This gives $\delta=\delta'$, the formula for $\delta$, and the number of classes. If all $P^{(z)}$ are even, every cell is symmetric (note $P^{(-z)}=P^{(z)}$ by (F1)), so the scheme is symmetric by Lemma~\ref{lem:reflexive}. Finally, the classes found in (i)--(iv) are $\{(a,b,c):b\in L\setminus\{0\},K(a,b,c)=Y\}$, $\{(a,0,c):a\ne0,G^{(0)}(a)=Y\}$, $\{(0,0,c):c\in L\setminus\{0\}\}$, $\{0\}$, which is $\Om_s(K)$ with the second and third coordinates interchanged (the cut variable of $K$ is $v$).
\end{proof}

Fusions of $\Om_s(H)$ along $\F_r$-subspaces of $\F_q$ are again Fourier-reflexive if the sign patterns of the ingredients are compatible with the fusion. For an $\F_r$-subspace $\Kz$ of $\F_q$ with $\Kp$ as in Section~\ref{sec:H} let
\begin{equation}\label{eq:Omega}
\Om(H,\Kz)=\{\{0\},\ B_1,\ B_2,\ e_i,\ d_i,\ c_i\ (i\in\F_r)\},
\end{equation}
where $c_i=\bigcup_{L\not\subseteq\Kz}c_{L,i}$, $d_i=\bigcup_{L\subseteq\Kz}c_{L,i}$, $B_1=\bigcup_{L\subseteq\Kp}B_L$, $B_2=\bigcup_{L\not\subseteq\Kp}B_L$, and empty cells are omitted ($c_i,B_1=\emptyset$ iff $\Kz=\F_q$; $d_i,B_2=\emptyset$ iff $\Kz=\{0\}$). Thus $\Om(H,\Kz)$ is the partition into the level sets of $H$ cut according to $z\in\Ko:=\F_q\setminus\Kz$, $z\in\Kz\setminus\{0\}$, $z=0$, with $\{0\}\times\F_q\times\{0\}$ split along $\Kp$; for $\Kz=\{0\}$ and $s=1$ it is the partition of \cite[Theorem~1]{AKMO}.

\begin{theorem}\label{thm:K0}
Let $n$ be even with $s\le n/2$, let $\Kz$ be an $\F_r$-subspace of $\F_q$, and let $(P^{(z)})$ satisfy (F1), (F2), (F4) and (F3) with respect to $\Kz$. Then $\Om(H,\Kz)$ is Fourier-reflexive with dual partition $\Om(K,\Kz)$ (coordinates interchanged as above), and it induces a symmetric association scheme with
\[
|\Om(H,\Kz)|-1=\begin{cases} 2r+\delta & \text{if }\Kz=\{0\}\text{ or }\Kz=\F_q,\\ 3r+1+\delta&\text{otherwise,}\end{cases}
\]
$\delta$ as in Theorem~\ref{thm:main}, provided all $P^{(z)}$ are even.
\end{theorem}
\begin{proof}
Summing \eqref{eq:chic} over the lines in $\Kz$, respectively not in $\Kz$, and using (F3) and $\sum_{L\subseteq\Kz}(r[c\in L^\perp]-1)=|\Kz|[c\in\Kp]-1$ (count the lines of $\Kz$ contained in the hyperplane $c^\perp$), one obtains
\begin{align*}
\chi_\xi(c_i)&=[b\in\Ko]\,\tfrac qr\,p^{n/2}\KL^{(1)}\big(Y_b(\xi),i\big)+\delta_{a,0}\delta_{b,0}\,\tfrac qr\,Q\big(q\delta_{c,0}-|\Kz|[c\in\Kp]\big),\\
\chi_\xi(d_i)&=[b\in\Kz\setminus\{0\}]\,\tfrac qr\,p^{n/2}\KL^{(2)}\big(Y_b(\xi),i\big)+\delta_{a,0}\delta_{b,0}\,\tfrac qr\,Q\big(|\Kz|[c\in\Kp]-1\big),\\
\chi_\xi(B_1)&=|\Kp|[b\in\Kz]-1,\qquad \chi_\xi(B_2)=q\delta_{b,0}-|\Kp|[b\in\Kz],
\end{align*}
with $\KL^{(1)},\KL^{(2)}$ the sums for the common sign patterns on $\Ko$ and $\Kz\setminus\{0\}$. The regions $b\in\Ko$, $b\in\Kz\setminus\{0\}$, ($b=0,a\ne0$), ($a=b=0,c\ne0$), $\xi=0$ then split into $r$, $r$, $r-1+\delta'$, $[\Kz\ne\F_q]+[\Kz\ne\{0\}]$ (according to $c\in\Kp$ or not), and $1$ classes respectively, by the same arguments as in the proof of Theorem~\ref{thm:main}; they are separated by $\chi_\xi(B_1)$ (which is $-1$ exactly for $b\in\Ko$), $\chi_\xi(B_2)$ (which is $-|\Kp|$ exactly for $b\in\Kz\setminus\{0\}$), $\chi_\xi(c_i)$ or $\chi_\xi(d_i)$ (zero in the region $b=0,a\ne0$, nonzero for $a=b=0,c\ne0$), and $|c_i|,|d_i|>0$. Counting the cells of $\Om(H,\Kz)$ as in the proof of Theorem~\ref{thm:main} gives $|\Om(H,\Kz)|=|\widehat{\Om(H,\Kz)}|$.
\end{proof}

\begin{remark}\label{rem:main}
\begin{enumerate}
\item[(a)] Theorem~\ref{thm:main} needs no condition on the sign patterns of the ingredients: they may differ from line to line, and $H$ may have only non-weakly regular components, only weakly regular ones, or both. Nor does it need $H$ to be vectorial dual-bent: the ingredient $P^{(0)}$ may be vectorial dual-bent but not of inversion type, in which case $H$ need not be vectorial dual-bent (and is not, in the two- or three-ingredient case, by Proposition~\ref{prop:equiv}); see Example~\ref{ex:PS}. The inversion condition is needed only for the ingredients on $z\ne0$, where the term $\Tr^k_s(yz)$ forces the scaling $t^{-1}$ (Lemma~\ref{lem:walsh}). Condition (F3) enters only for the fusions along $\Kz$ in Theorem~\ref{thm:K0}, where the cells $d_i$, $c_i$ do not distinguish the lines in $\Kz$, resp.\ outside $\Kz$, while the dual partition does when the sign patterns differ.
\item[(b)] Splitting the sets $B_L$ (resp.\ $B_1,B_2$ in Theorem~\ref{thm:K0}) off the level set of $0$ is the mechanism of \cite[Lemma~13]{AKMO3}, by which a subspace contained in a cell of a Fourier-reflexive partition is made a cell of its own; there it yields the $(p^m+1)$-class schemes of \cite[Theorem~4(ii)]{AKMO3}. Likewise, the degenerate case $\delta=0$ ($n=2s$, all components of $P^{(0)}$ weakly regular but not regular, so that $P^{(0)}$ has no nonzero zero) is the phenomenon behind the $(p^m-1)$-class schemes in \cite[Theorem~4(i)]{AKMO3}.
\item[(c)] The cell sizes are $|c_{L,i}|=\frac qr Q(r-1)$, $|B_L|=r-1$, $|e_i|=q(p^{n-s}+p^{n/2-s}\KL^{(0)}(0,i))-q\delta_{i,0}$; for $P^{(0)}$ with components of the same type $\varepsilon_0$, $\KL^{(0)}(0,i)=\varepsilon_0(r\delta_{i,0}-1)$.
\item[(d)] The character sums are, up to the factor $\frac qr p^{n/2}$, values of the Kloosterman-type sums $\KL_\varepsilon(Y,i)$ over $\F_r$, which are not rational in general. This is in contrast with the scalar constructions \cite{OP22,AKMO}, where $\ell$-form conditions guarantee integrality; here the counting of the dual partition rests on Lemma~\ref{lem:injective} instead.
\item[(e)] Evenness is used only for the symmetry of the scheme. The proofs do not use that $H$ is bent, only \eqref{eq:walsh} for the components of the ingredients. They use the $\F_r$-linear structure of the cutting sets in two places: $\F_r^*$-stability (so that $b/t\in L\Leftrightarrow b\in L$) and additivity (so that the sums $\sum_{z\in L}\eq(-cz)$ take only two values). Example~\ref{ex:K0} shows that for a non-subspace $\Kz$ the partition $\Om(H,\Kz)$ is not Fourier-reflexive.
\item[(f)] With $\Tr^k_s(yz^{l-1})$, $\gcd(l-1,q-1)=1$, in place of $\Tr^k_s(yz)$ one has $z=(t^{-1}v)^{d-1}$ in Lemma~\ref{lem:walsh}, $(l-1)(d-1)\equiv1\bmod(q-1)$, and the same proofs apply if $\sigma(t)=t^{1-d}$ replaces the inversion in (F2) and the cutting sets are stable under $z\mapsto cz^{d-1}$. We do not pursue this.
\end{enumerate}
\end{remark}

\section{Changing the level: fusion schemes}\label{sec:fusion}

Let $s'\mid s$, $r'=p^{s'}$, and $\gamma\in\F_r^*$. Then
\[
H'=\Tr^s_{s'}(\gamma H)\colon V\to\F_{r'},\qquad H'(x,y,z)=\Tr^s_{s'}\big(\gamma P^{(z)}(x)\big)+\Tr^k_{s'}(\gamma yz),
\]
and $\Kz$ is also an $\F_{r'}$-subspace. Let $\Om'=\Om(H',\Kz)$ be the partition \eqref{eq:Omega} at level $s'$, i.e.\ with cells $c'_j=\{z\in\Ko:H'=j\}$, $d'_j$, $e'_j=\{(x,y,0):x\ne0,\Tr^s_{s'}(\gamma P^{(0)}(x))=j\}$, $j\in\F_{r'}$, and $B_1,B_2$ as before ($\Kp$ does not depend on the level).

\begin{theorem}\label{thm:fusion}
In the situation of Theorem~\ref{thm:K0}, suppose in addition that $P^{(0)}$ is of inversion type ($\sigma_0=\iota$ after normalization of $G^{(0)}$). Then for every $s'\mid s$ and $\gamma\in\F_r^*$:
\begin{enumerate}
\item[(i)] $\Om'$ is a fusion of $\Om$: $c'_j=\bigcup_{\Tr^s_{s'}(\gamma i)=j}c_i$, and likewise for $d'_j$, $e'_j$.
\item[(ii)] $\Om'$ is Fourier-reflexive and induces a fusion scheme of the scheme of Theorem~\ref{thm:K0}, with $2r'+\delta'$ or $3r'+1+\delta'$ classes according to Theorem~\ref{thm:K0}, where $\delta'=\delta$ if $s'=s$ and $\delta'=1$ if $s'<s$ (as then $n\ge2s>2s'$).
\item[(iii)] If $\gamma_1,\gamma_2\in\F_r^*$ lie in different cosets of $\F_{r'}^*$, then $\Om'_{\gamma_1}\ne\Om'_{\gamma_2}$. Hence the scheme of Theorem~\ref{thm:K0} has, for every $s'\mid s$, at least $(r-1)/(r'-1)$ different fusion schemes of this kind.
\end{enumerate}
\end{theorem}
\begin{proof}
(i) is clear from $\{H'=j\}=\bigcup_{\Tr^s_{s'}(\gamma i)=j}\{H=i\}$.
(ii) Let $\phi\colon V\to V$, $\phi(x,y,z)=(x,\gamma^{-1}y,z)$, and $H''=H'\circ\phi$, so $H''(x,y,z)=P'^{(z)}(x)+\Tr^k_{s'}(yz)$ with $P'^{(z)}=\Tr^s_{s'}(\gamma P^{(z)})$. The family $(P'^{(z)})$ satisfies (F1)--(F4) at level $s'$: (F1) as $\F_{r'}^*\subseteq\F_r^*$; the components are $P'^{(z)}_t=\Tr^{s'}_1(t\Tr^s_{s'}(\gamma P^{(z)}))=P^{(z)}_{t\gamma}$, $t\in\F_{r'}^*$, weakly regular with sign $\varepsilon_{z,t\gamma}$, so (F3) holds; and, since all $P^{(z)}$ including $P^{(0)}$ are of inversion type, $(P'^{(z)}_t)^*=\Tr^s_1((t\gamma)^{-1}G^{(z)})=\Tr^{s'}_1\big(t^{-1}\Tr^s_{s'}(\gamma^{-1}G^{(z)})\big)$, so (F2) holds with $G'^{(z)}=\Tr^s_{s'}(\gamma^{-1}G^{(z)})$ (this is where the inversion hypothesis on $P^{(0)}$ is used: for a general $\sigma_0$ the projection $P'^{(0)}$ need not be vectorial dual-bent, cf.\ \cite[Corollary~15]{AKMO3}), and (F4) is inherited. Hence $\Om(H'',\Kz)$ is Fourier-reflexive by Theorem~\ref{thm:K0}. Since $\gamma^{-1}\Kp=\Kp$, $\phi$ maps the cells of $\Om(H'',\Kz)$ onto the cells of $\Om'$: $\Om'=\phi(\Om(H'',\Kz))$. A linear bijection $\phi$ transforms the dual partition by the adjoint of $\phi^{-1}$, so $\Om'$ is Fourier-reflexive with the same number of cells, and the scheme induced by $\Om'$ is isomorphic to the one induced by $\Om(H'',\Kz)$. It is a fusion of the scheme of $\Om$ by (i). (The computation of the components of $P'^{(z)}$ is the one used in \cite[Section~4]{AKMO3} for projections $\Tr^m_s\circ F$ of vectorial dual-bent functions.)
(iii) Suppose $\Kz\ne\F_q$ (for $\Kz=\F_q$ argue with $d'_j$ instead). The cells $c'_j$ are unions of level sets of $H'=\Tr^s_{s'}(\gamma H)$ on $V_n\times\F_q\times\Ko$, on which $H$ is surjective onto $\F_r$ (already on each slice $z\in\Ko$, via $y$). If $\gamma_1,\gamma_2$ are $\F_{r'}$-linearly independent, there is $X\in\F_r$ with $\Tr^s_{s'}(\gamma_1X)=0\ne\Tr^s_{s'}(\gamma_2X)$; two points with $H$-values differing by $X$ lie in the same cell of $\Om'_{\gamma_1}$ but in different cells of $\Om'_{\gamma_2}$, cf.\ \cite[Proposition~5]{AKMO2}.
\end{proof}

The partitions $\Om(H,\Kz)$ are fusions of $\Om_s(H)$, and by Theorem~\ref{thm:fusion} the level-$s'$ partitions $\Om(\Tr^s_{s'}(\gamma H),\Kz)$ are fusions of them as well. The full partitions $\Om_s(H)$ and $\Om_{s'}(\Tr^s_{s'}(\gamma H))$ at different levels are, however, not comparable in general: the first cuts $z$ along $\F_r$-lines and the second along the finer $\F_{r'}$-lines, while the first has the finer level sets. They have the partitions $\Om(\Tr^s_{s'}(\gamma H),\Kz)$, $\Kz$ an $\F_r$-subspace, as common fusions.

For $s=1$ the partition $\Om_1(H)$ cuts $z\ne0$ and $y\ne0$ along the cosets of $\F_p^*$: in the notation of \cite{AK}, $c_{L,i}=D_{F,i,v}$ and $B_L=A_v$ for $L=v\F_p$, and $e_i=D_{F,i,0}$. Theorem~\ref{thm:main} gives $\frac{p^k-1}{p-1}(p+1)+p-1+\delta$ classes, and by Proposition~\ref{prop:s1} its hypotheses (F1), (F2) at $s=1$ are exactly the conditions (C1), (C2) of \cite[Theorem~1]{AK} for $l=2$: $f^{(z)}$, $z\ne0$, weakly regular $2$-forms with $f^{(cz)}=f^{(z)}$ for $c\in\F_p^*$, and $f^{(0)}$ a weakly regular $\ell$-form for some $\ell$. Thus Theorem~\ref{thm:main} with $s=1$ is the case $l=2$, $n$ even of \cite[Theorem~1]{AK} (which also covers general $l$, odd $n$, and $n=1$), proved there by a different route, and for $s>1$ it provides the $\F_{p^s}$-valued analogues. The coarser fusion $\Om(H,\{0\})$ at $s=1$ is the partition of \cite[Theorem~1]{AKMO}:

\begin{corollary}\label{cor:AKMO}
Let $s=1$, $\Kz=\{0\}$, and let $f^{(z)}\colon\F_{p^n}\to\F_p$, $z\in\F_q$, be weakly regular bent functions with $f^{(cz)}=f^{(z)}$ for $c\in\F_p^*$, all of the same type and $2$-forms for $z\ne0$, $f^{(0)}$ an $\ell$-form, with $f^{(0)}(0)=(f^{(0)})^*(0)=0$, and $n$ even. Then $F(x,y,z)=f^{(z)}(x)+\Tr^k_1(yz)$ and the partition $\Om=\{\{0\},B\setminus\{0\},D_{F,i,\mp}\}$ of \cite[Theorem~1]{AKMO} induce a $(2p+1)$-class association scheme, resp.\ a $2p$-class scheme if $n=2$ and $f^{(0)}$ is of type $(-)$.
\end{corollary}

This is \cite[Theorem~1]{AKMO} for $l=2$ and $n$ even (their condition (C3) with $\tilde l=2$ is the $2$-form condition, which by Proposition~\ref{prop:s1} is the inversion condition). In the other direction, Theorem~\ref{thm:fusion} with $s=k$ and $s'=1$ shows that every $(2q+1)$-class scheme of Theorem~\ref{thm:main} for which also $P^{(0)}$ is of inversion type has $(q-1)/(p-1)$ different $(2p+1)$-class fusion schemes of the type of \cite{AKMO}. In Section~\ref{sec:examples} we show that the flexibility of the scalar setting (ingredients varying between the cosets of $\F_p^*$) does not survive the passage to $\F_q$-valued level sets.

\section{Explicit families, examples, and the role of the hypotheses}\label{sec:examples}

We first record two infinite families to which Theorem~\ref{thm:main} applies and for which all components of $H$ are non-weakly regular; they are obtained from Example~\ref{ex:ingredients}.

\begin{corollary}\label{cor:family}
Let $p$ be an odd prime, $s\mid k$, $q=p^k$, $r=p^s$, $V_n=\F_{p^n}$ with the trace form, and let $H$ be given by \eqref{eq:H} with one of the following choices.
\begin{enumerate}
\item[(a)] $n=2m$ with $s\mid m$, and either $p\equiv1\bmod4$, or $p\equiv3\bmod4$ and $n\equiv0\bmod4$; $P^{(0)}(x)=\Tr^n_s(x^2)$, and $P^{(z)}(x_1,x_2)=\Tr^m_s(x_1x_2)$ for $z\ne0$ (with $\F_{p^n}$ identified with $\F_{p^m}\times\F_{p^m}$ by an $\F_{p^m}$-basis).
\item[(b)] $n$ even, $s\mid n$ with $n/s$ odd and $n/s\ge3$, and $\nu$ a nonsquare in $\F_r$; $P^{(0)}(x)=\Tr^n_s(x^2)$ and $P^{(z)}(x)=\Tr^n_s(\nu x^2)$ for $z\ne0$.
\end{enumerate}
Then all $r-1$ components of $H\colon\F_{p^n}\times\F_q\times\F_q\to\F_r$ are non-weakly regular bent, $H$ is vectorial dual-bent, and $\Om_s(H)$ induces a symmetric association scheme with $\frac{q-1}{r-1}(r+1)+r-1+\delta$ classes, where $\delta=0$ in case (a) with $n=2s$, and $\delta=1$ otherwise. Moreover, for every $\F_r$-subspace $\Kz$ of $\F_q$, $\Om(H,\Kz)$ induces a symmetric association scheme with $2r+\delta$ classes if $\Kz\in\{\{0\},\F_q\}$ and $3r+1+\delta$ classes otherwise, and all these schemes have fusion schemes at every level $s'\mid s$ as in Theorem~\ref{thm:fusion}.
\end{corollary}
\begin{proof}
All ingredients are even, vanish at $0$, are of inversion type by Example~\ref{ex:ingredients}(a), (b), and satisfy $W_{P_t}(0)=\varepsilon_t p^{n/2}$ (the quadratic functions $\Tr^n_1(t\lambda x^2)$ have $W(0)=\varepsilon p^{n/2}$ by \cite[Corollary~3]{HK}, and Maiorana--McFarland functions are regular with $W(0)=p^{n/2}$); so (F1), (F2), (F4) hold, and (F3) holds trivially as the ingredient is the same for all $z\ne0$. In (a), $n/s$ is even, so by Example~\ref{ex:ingredients}(a) the sign pattern of $P^{(0)}$ is constant equal to $(-1)^{n-1}=-1$ if $p\equiv1\bmod4$ and to $(-1)^{n-1}i^n=-1$ if $p\equiv3\bmod4$ and $n\equiv0\bmod4$, while all components of $P^{(1)}$ are regular; hence $\varepsilon_{0,t}\ne\varepsilon_{1,t}$ for all $t$ and every $H_t$ is non-weakly regular by Corollary~\ref{cor:nwr}. In (b), $n/s$ is odd and the sign patterns of $P^{(0)}$ and $P^{(1)}$ are $\pm\eta_r(t)$ and $\pm\eta_r(\nu t)=\mp\eta_r(t)$, again opposite for every $t$. The bound $s\le n/2$ holds in both cases ($s\le n/3$ in (b)), and $n=2s$ is possible only in (a). The remaining statements are Theorems~\ref{thm:main}, \ref{thm:K0} and \ref{thm:fusion}, and Proposition~\ref{prop:equiv}.
\end{proof}

All numerical statements in this section were verified by computer; the Python scripts, which fix all parameters, are provided as supplementary material. Character sums over $\F_3$-spaces are computed as follows: for a cell $U$ and $\xi\in V$, $\chi_\xi(U)=n_0+n_1\zeta_3+n_2\zeta_3^2=(n_0-n_2)+(n_1-n_2)\zeta_3$ with $n_j=|\{u\in U:\ip{\xi}{u}=-j\}|$, so the sum is determined by two integers. For all partitions of $\F_3^8$ occurring in Examples~\ref{ex:n4}, \ref{ex:K0}, \ref{ex:PS} and \ref{ex:coset} these integers were computed exactly, by integer arithmetic on the $3^8\times3^8$ matrix of inner products, so the values of $|\Om|$ and $|\Omd|$ stated there are exact; for Example~\ref{ex:n4} the formulas of Lemma~\ref{lem:charsums} were also checked for all $\xi\in V$, with the trace form as inner product. For the larger examples on $\F_3^{10}$ and $\F_3^{12}$ (Examples~\ref{ex:n6}, \ref{ex:n8} and \ref{ex:nwr}) the counts were obtained numerically, by a double precision fast Fourier transform followed by rounding to the nearest point of the lattice $\Z+\Z\zeta_3$, with the check that every computed value lies within $10^{-4}$ of a lattice point (the accumulated rounding error of the transform is several orders of magnitude smaller); these statements are therefore numerically verified rather than formally certified.

In all examples $p=3$, $k=2$, $q=9$, and $\F_9=\F_3[t]/(t^2+1)$. Further, $\F_{81}=\F_3[u]/(u^4+u^3+u^2+1)$ with $\F_9$ embedded via $t\mapsto u+2u^2$; $\F_{3^6}=\F_9[w]/(w^3+2w^2+1)$; $\F_{3^8}=\F_9[w]/(w^4+w^3+tw^2+1)$; $\F_{3^n}$ is identified with $\F_9^{n/2}$ through these bases, and the inner product used for the dual partitions is the standard dot product on $\F_3^{n+4}$ (any nondegenerate form gives an isomorphic scheme).

\begin{example}[$s=k=2$, $\Kz=\{0\}$, $n=2k$]\label{ex:n4}
Let $n=4$, $P^{(0)}(x)=\Tr^4_2(x^2)$ on $\F_{81}$ (all components weakly regular but not regular, Example~\ref{ex:ingredients}(a)), and $P^{(z)}(x_1,x_2)=x_1x_2$ on $\F_9\times\F_9$ for $z\ne0$ (all components regular). Then $H\colon\F_3^8\to\F_9$ has eight non-weakly regular components, each with $|B_-(H_t)|=729$, and $H$ is vectorial dual-bent. The partition into the $\F_9$-valued level sets of $H$ (with $\{0\}$ split off) has $10$ cells and a dual partition with $19$ cells, so it is not Fourier-reflexive. The partition $\Om(H,\{0\})$ has $e_0=\emptyset$ ($\delta=0$), hence $19$ cells of sizes $1,8,90^{\times8},648^{\times9}$, and $|\Omd|=19$: an $18$-class symmetric scheme on $\F_3^8$. Its level-$1$ fusions $\Om(\Tr^2_1(\gamma H),\{0\})$, one for each of the four cosets $\gamma\F_3^*$, have $8$ cells of sizes $1,8,180,270,270,1944^{\times3}$ and are the $7$-class schemes of \cite[Theorem~1]{AKMO}. Moreover, for each of the four $\F_3$-lines $\Kz\subset\F_9$ the partition $\Om(\Tr^2_1(H),\Kz)$ has $12$ cells and is Fourier-reflexive, giving four $11$-class fission schemes of each $7$-class scheme (Theorem~\ref{thm:K0}, $3r+1+\delta=11$), and the full level-$1$ partition $\Om_1(\Tr^2_1(H))$ has $20$ cells and is Fourier-reflexive, a $19$-class scheme as predicted by Theorem~\ref{thm:main} ($\frac{q-1}{r-1}(r+1)+r=4\cdot4+3$). The two $19$-class schemes $\Om_2(H)$ and $\Om_1(\Tr^2_1(H))$ are different partitions of $\F_3^8$ with the same number of classes.
\end{example}

\begin{example}[$s=k=2$, $\Kz=\{0\}$, $n>2k$]\label{ex:n8}
Let $n=8$, $P^{(0)}(x)=\Tr^8_2(x^2)$ on $\F_{3^8}$, and $P^{(z)}(x_1,x_2)=\Tr^4_2(x_1x_2)$ on $\F_{81}\times\F_{81}$ for $z\ne0$. Then $H\colon\F_3^{12}\to\F_9$ has eight non-weakly regular components, $|e_0|=5904$, and $\Om(H,\{0\})$ has $20$ cells of sizes $1,8,5904,6642^{\times8},52488^{\times9}$ (Remark~\ref{rem:main}(c)) with $|\Omd|=20$: a $19$-class scheme. If $B_1$ is merged with $e_0$, the partition has $19$ cells but $|\Omd|=20$; the splitting of $B_1$ cannot be omitted, as in \cite[Theorem~1]{AKMO}.
\end{example}

\begin{example}[mixed sign patterns of the ingredients]\label{ex:n6}
Let $n=6$, $P^{(0)}(x)=\Tr^6_2(x^2)$ and $P^{(z)}(x)=\Tr^6_2(\nu x^2)$, $z\ne0$, on $\F_{3^6}$ with $\nu$ a nonsquare in $\F_9$. As $n/s=3$ is odd, $P^{(0)}_t$ is regular for the squares $t\in\F_9^*$ and weakly regular but not regular for the nonsquares, and $P^{(1)}$ has the opposite pattern. Hence all eight components of $H\colon\F_3^{10}\to\F_9$ are non-weakly regular, and $\Om(H,\{0\})$ has $20$ cells (sizes $1,8,720,648^{\times4},810^{\times4},5832^{\times9}$) with $|\Omd|=20$, a $19$-class scheme. Here $|e_i|$ takes two values for $i\ne0$, according to the quadratic character of $i$, since $\KL^{(0)}(0,i)$ depends on $i$ when the sign pattern is not constant.
\end{example}

\begin{example}[$\{0\}\ne\Kz\ne\F_q$ with an ingredient of the other type on $\Kz\setminus\{0\}$]\label{ex:K0}
Let $n=4$ and $\Kz=\F_3\subset\F_9$ (or any other $\F_3$-line), $P^{(z)}=\Tr^4_2(x^2)$ for $z\in\Kz$ and $P^{(z)}=x_1x_2$ for $z\notin\Kz$. All eight components of $H\colon\F_3^8\to\F_9$ are non-weakly regular. At level $s=1$ the full partition $\Om_1(\Tr^2_1(H))$ is Fourier-reflexive with $20$ cells although the sign patterns differ from line to line, as Theorem~\ref{thm:main} allows, and the fusion $\Om(\Tr^2_1(H),\Kz)$ has $12$ cells (sizes $1,2,6,180,270,270,486^{\times3},1458^{\times3}$, in accordance with Remark~\ref{rem:main}(c)) and $|\Omd|=12$: an $11$-class scheme. The partition of \cite[Theorem~1]{AKMO}, which cuts only at $z=0$, has $8$ cells and $|\Omd|=11$; it is not Fourier-reflexive, in accordance with the fact that hypothesis (C1) of \cite{AKMO} (all $f^{(z)}$, $z\ne0$, of the same type) is violated. At level $s=2$, $\Om(H,\Kz)$ is not Fourier-reflexive ($29$ cells, $|\Omd|=119$), as $\F_3$ is not an $\F_9$-subspace; and if, for the same $H$, the cutting set $\Kz$ in \eqref{eq:Omega} is replaced by the non-subspace $\{0,1\}$, the resulting level-$1$ partition has $12$ cells and $|\Omd|=16$.
\end{example}

The remaining examples show that the conclusion of Theorem~\ref{thm:main} may fail when one of its hypotheses is omitted.

\begin{example}[failure without inversion type]\label{ex:PS}
Let $n=4$, $\Kz=\{0\}$, $P^{(0)}$ as in Example~\ref{ex:n4} or $P^{(0)}(x)=\Tr^4_2(x^{14})$ (a projection to $\F_9$ of the Coulter--Matthews planar function $x^{14}$ on $\F_{81}$; the eight components $\Tr^4_1(\alpha x^{14})$, $\alpha\in\F_9^*$, of this projection are all weakly regular but not regular, whereas among all $80$ components of $x^{14}$ half are regular), and $P^{(z)}(x_1,x_2)=x_1x_2^{7}=x_1/x_2$, $z\ne0$, the vectorial partial spread function on $\F_9\times\F_9$, all of whose components are regular. $P^{(1)}$ is vectorial dual-bent with $\sigma(t)=t$, and one checks that there is no $G$ with $(P^{(1)}_t)^*=\Tr^2_1(t^{-1}G)$ for all $t\in\F_9^*$. For both choices of $P^{(0)}$ all components of $H$ are non-weakly regular and dual-bent, $H$ is not vectorial dual-bent (Proposition~\ref{prop:equiv}), and $\Om(H,\{0\})$ has $19$ cells with $|\Omd|=91$, resp.\ $107$. At level $s=1$, however, the inversion condition over $\F_3$ is automatic (Example~\ref{ex:ingredients}(c)), and $\Om(\Tr^2_1(\gamma H),\{0\})$ induces the $7$-class scheme of Corollary~\ref{cor:AKMO}. The roles of the two ingredients are not symmetric: with $P^{(0)}(x_1,x_2)=x_1/x_2$ (the partial spread function, $\sigma_0(t)=t$) and $P^{(z)}=\Tr^4_2(x^2)$ for $z\ne0$, all components of $H$ are again non-weakly regular, $H$ is not vectorial dual-bent (Proposition~\ref{prop:equiv}), and yet $\Om_2(H)$ is Fourier-reflexive with $20$ cells (sizes $1,8,144,72^{\times8},648^{\times9}$), a $19$-class scheme, as Theorem~\ref{thm:main} predicts. In particular, Fourier-reflexivity of $\Om_k(H)$ does not force $H$ to be vectorial dual-bent. In Lemma~\ref{lem:charsums} the failure is visible: with $(P^{(1)}_t)^*=\Tr^k_1(\sigma(t)G^{(1)})$ the sum in \eqref{eq:chic} depends on the pair $(G^{(1)}(a),\Tr^k_s(bc))$ rather than on $Y_b(\xi)$ alone, unless $\sigma(t)=t^{-1}$, so region (i) splits into up to $q^2$ classes.
\end{example}

\begin{example}[failure when ingredients vary within a coset of $\F_r^*$]\label{ex:coset}
Let $n=4$, $\Kz=\{0\}$, $P^{(0)}(x)=\Tr^4_2(x^2)$, and $P^{(z)}(x_1,x_2)=\lambda_zx_1x_2$ for $z\ne0$, where $\lambda_z\in\F_9^*$ is constant on the cosets of $\F_3^*$ in $\F_9^*$ and takes different values on different cosets. Each $P^{(z)}$ is of inversion type with regular components, so (F1), (F2), (F4) hold at level $s=1$ (with $r=3$) but (F1) fails at level $s=2$. Accordingly, $\Om_1(\Tr^2_1(H))$ is Fourier-reflexive ($20$ cells, a $19$-class scheme as in \cite{AK}) and so is its fusion $\Om(\Tr^2_1(H),\{0\})$ ($8$ cells, the $7$-class scheme of \cite[Theorem~1]{AKMO}), whereas $H$ is not vectorial dual-bent and $\Om(H,\{0\})$ has $19$ cells with $|\Omd|=91$. The reason is the one visible in Lemma~\ref{lem:charsums}: the ingredient $P^{(b/t)}$ entering $\chi_\xi(c_i)$ must be independent of $t\in\F_r^*$, which for $r=q$ means constant on $\F_q^*$. The scalar setting of \cite{AKMO} only has $t\in\F_p^*$, hence only requires constancy on $\F_p^*$-cosets (their condition (C2)).
\end{example}

\begin{example}[ingredients with non-weakly regular components]\label{ex:nwr}
Let $n=8$, $\Kz=\{0\}$, $P^{(0)}(x)=\Tr^8_2(x^2)$, and let $P^{(z)}$, $z\ne0$, be the function $H$ of Example~\ref{ex:n4}, viewed as a function $\F_3^8\to\F_9$: it is vectorial dual-bent of inversion type, and all its components are non-weakly regular with the same sign set (Corollary~\ref{cor:nwr}), so only the weak regularity in (F2) is violated. The resulting $H'\colon\F_3^{12}\to\F_9$ is vectorial bent, but $\Om(H',\{0\})$ has $20$ cells and $|\Omd|=29$. The number $29=1+1+9+18$ is what Lemma~\ref{lem:charsums} predicts when $\varepsilon_{1,t}$ is replaced by $\varepsilon_{1,t}\theta(a)$ with $\theta(a)=\pm1$ depending on $a$: region (i) splits according to $(\theta(a),Y_b(\xi))$ into $2q$ classes. Moreover, the natural attempt to repair this by refining does not stabilize quickly: after four steps of the iteration $\Om_{j+1}=\Om_j\wedge\widehat{\widehat{\Om_j}}$ the cell numbers are $20,29,37,100$ and no Fourier-reflexive partition has been reached. Thus the construction suggested in \cite[Remark~7]{AKMO} does not extend verbatim, in general, to ingredients with non-weakly regular components. (The value $|\Omd|=29$ is explained by Lemma~\ref{lem:charsums} as indicated; the cell numbers $37$ and $100$ of the further refinements are numerical observations.)
\end{example}

\section{Concluding remarks}

Theorem~\ref{thm:main} gives a uniform sufficient condition, vectorial dual-bentness of the ingredients, of inversion type for those on $z\ne0$, under which the generalized Maiorana--McFarland function \eqref{eq:H} induces association schemes at every level $s\mid k$, and Theorems~\ref{thm:K0} and \ref{thm:fusion} describe fusions along subspaces and (when all ingredients are of inversion type) across levels. Several questions remain.
\begin{enumerate}
\item Example~\ref{ex:PS} shows that Fourier-reflexivity of $\Om_k(H)$ does not force $H$ to be vectorial dual-bent: what matters is the inversion type of the ingredients on $z\ne0$. Whether this inversion condition is necessary for the Fourier-reflexivity of $\Om_k(H)$ is open; in the two counterexamples at level $k$ (Examples~\ref{ex:PS}, \ref{ex:coset}) it fails together with reflexivity. For the preimage set partition of Maiorana--McFarland functions $\Tr^m_1(x\phi(y))$ the necessity of vectorial dual-bentness has been established in \cite{AKMO3}.
\item Besides the fusions of Theorems~\ref{thm:K0} and \ref{thm:fusion}, one may fuse along subgroups of $\F_r^*$ (squares and nonsquares, as in \cite{WWF,AKMO,AK}). A description of all fusion schemes, and the determination of which of the resulting schemes are isomorphic (e.g.\ via $p$-ranks of adjacency matrices, as in \cite{AKMO2}), is open.
\item The ingredients in Example~\ref{ex:ingredients} are quadratic or Maiorana--McFarland. Other vectorial dual-bent functions of inversion type would give further examples; the partial spread functions are not of this kind. For $l\ne2$, the condition $\sigma(t)=t^{1-d}$ of Remark~\ref{rem:main}(f) should relate to the $\tilde l$-form conditions of \cite{AKMO,AK} as in Proposition~\ref{prop:s1}; we have not pursued this.
\item Similar $\F_q$-valued refinements should exist for the vectorial semi-direct sum construction of \cite{CMP20}. For the vectorial generalized Rothaus construction of \cite{CM24} the components are in general not dual-bent, and the present method, which relies on the vectorial dual $K=H^*$, does not apply; whether non-dual-bent (vectorial) bent functions can give rise to association schemes at all is, to our knowledge, open.
\end{enumerate}


\begin{thebibliography}{99}

\bibitem{AK} N.~Anbar, T.~Kalayc\i, Further results on association schemes from non-weakly regular bent functions, Cryptogr.\ Commun.\ (2026), https://doi.org/10.1007/s12095-026-00909-8.

\bibitem{AKM} N.~Anbar, T.~Kalayc\i, W.~Meidl, Amorphic association schemes from bent partitions, Discrete Math.\ 347 (2024), 113658.

\bibitem{AKMO} N.~Anbar, T.~Kalayc\i, W.~Meidl, F.~\"Ozbudak, $(2p+1)$-class association schemes from the generalized Maiorana--McFarland class, Finite Fields Appl.\ 103 (2025), 102568.

\bibitem{AKMO2} N.~Anbar, T.~Kalayc\i, W.~Meidl, F.~\"Ozbudak, Bent partitions and Maiorana--McFarland association schemes, Cryptogr.\ Commun.\ 17 (2025), 1641--1657.

\bibitem{AKMO3} N.~Anbar, T.~Kalayc\i, W.~Meidl, F.~\"Ozbudak, Vectorial dual-bent functions and association schemes, Des.\ Codes Cryptogr.\ (2026), Paper No.\ 205, https://doi.org/10.1007/s10623-026-01946-3.

\bibitem{BCN} A.~E.~Brouwer, A.~M.~Cohen, A.~Neumaier, Distance-Regular Graphs, Springer, Berlin, 1989.

\bibitem{CMP13} A.~\c{C}e\c{s}melio\u{g}lu, W.~Meidl, A.~Pott, On the dual of (non)-weakly regular bent functions and self-dual bent functions, Adv.\ Math.\ Commun.\ 7 (2013), 425--440.

\bibitem{CMP18} A.~\c{C}e\c{s}melio\u{g}lu, W.~Meidl, Bent and vectorial bent functions, partial difference sets, and strongly regular graphs, Adv.\ Math.\ Commun.\ 12 (2018), 691--705.

\bibitem{CMP20} A.~\c{C}e\c{s}melio\u{g}lu, W.~Meidl, A.~Pott, Vectorial bent functions in odd characteristic and their components, Cryptogr.\ Commun.\ 12 (2020), 899--912.

\bibitem{CM24} A.~\c{C}e\c{s}melio\u{g}lu, W.~Meidl, Vectorial bent functions with non-weakly regular components, IEEE Trans.\ Inform.\ Theory 70 (2024), 9214--9226.

\bibitem{GL} H.~Gluesing-Luerssen, Fourier-reflexive partitions and MacWilliams identities for additive codes, Des.\ Codes Cryptogr.\ 75 (2015), 543--563.

\bibitem{HK} T.~Helleseth, A.~Kholosha, Monomial and quadratic bent functions over the finite fields of odd characteristic, IEEE Trans.\ Inform.\ Theory 52 (2006), 2018--2032.

\bibitem{KSW} P.~V.~Kumar, R.~A.~Scholtz, L.~R.~Welch, Generalized bent functions and their properties, J.\ Combin.\ Theory Ser.\ A 40 (1985), 90--107.

\bibitem{OP22} F.~\"Ozbudak, R.~M.~Pelen, Imprimitive symmetric association schemes of classes 5 and 6 arising from ternary non-weakly regular bent functions, J.\ Algebraic Combin.\ 56 (2022), 635--658.

\bibitem{PTFL} A.~Pott, Y.~Tan, T.~Feng, S.~Ling, Association schemes arising from bent functions, Des.\ Codes Cryptogr.\ 59 (2011), 319--331.

\bibitem{TPF} Y.~Tan, A.~Pott, T.~Feng, Strongly regular graphs associated with ternary bent functions, J.\ Combin.\ Theory Ser.\ A 117 (2010), 668--682.

\bibitem{WF24} J.~Wang, F.-W.~Fu, New results on vectorial dual-bent functions and partial difference sets, Des.\ Codes Cryptogr.\ 91 (2023), 127--149.

\bibitem{WFWY} J.~Wang, F.-W.~Fu, Y.~Wei, J.~Yang, A further study of vectorial dual-bent functions, IEEE Trans.\ Inform.\ Theory 70 (2024), 7472--7483.

\bibitem{WWF} Y.~Wei, J.~Wang, F.-W.~Fu, Association schemes arising from non-weakly regular bent functions, Des.\ Codes Cryptogr.\ (2024), https://doi.org/10.1007/s10623-024-01495-7.

\bibitem{WHL} Y.~Wu, J.~Y.~Hyun, Y.~Lee, Characterization of $p$-ary functions in terms of association schemes and its applications, J.\ Combin.\ Theory Ser.\ A 187 (2022), 105576.

\end{thebibliography}
\end{document}